\documentclass{amsart}
\usepackage{setspace, amssymb, amsmath, amsfonts}

\theoremstyle{plain}
\newtheorem{thm}{Theorem}[section]
\newtheorem{theorem}[thm]{Theorem}

\newtheorem{corollary}[thm]{Corollary}

\newtheorem{lemma}[thm]{Lemma}

\newtheorem{proposition}[thm]{Proposition}

\theoremstyle{definition}
\newtheorem{defn}[thm]{Definition}
\newtheorem{definition}[thm]{Definition}

\theoremstyle{remark}
\newtheorem{remark}[thm]{Remark}

\newcommand{\okdvn}{\mathcal{O}_{\kappa, \cdot, v}^{\cdot, D, \cdot}(n)}
\newcommand{\iokdvn}{\mathcal{IO}_{\kappa,\cdot, v}^{\cdot, D, \cdot}(n)}
\newcommand{\akdvn}{\mathcal{A}_{\kappa,\cdot, v}^{\cdot, D, \cdot}(n)}
\newcommand{\akdn}{\mathcal{A}_{\kappa,\cdot}^{\cdot, D}(\leq n)}
\newcommand{\xo}{X_{O}}

\newcommand{\ccu}{(\wtu,\Gamma_U, \pi_U)}
\newcommand{\ccw}{(\widetilde{W}, \Gamma_W, \pi_W)}

\newcommand{\tp}{\tilde{p}}
\newcommand{\wtu}{\widetilde{U}}

\newcommand{\wtw}{\widetilde{W}}

\newcommand{\dpp}{\delta^{\prime\prime}}

\newcommand{\ve}{\varepsilon}
\newcommand{\uinpi}{U_i\backslash\{p_i\}}
\newcommand{\unp}{U\backslash\{p\}}
\newcommand{\uitnpi}{\wtu_i\backslash\{\tilde{p}_i\}}
\newcommand{\utnp}{\wtu\backslash\{\tilde{p}\}}

\begin{document}

\title[Orbifold homeomorphism finiteness based on geometric constraints]{Orbifold homeomorphism finiteness based on geometric constraints}

\author[E. Proctor]{Emily Proctor}

\address{Department of Mathematics, Middlebury College, Middlebury, VT 05753}

\email{eproctor@middlebury.edu}

\thanks{{\it Keywords:} Orbifold  \ Global Riemannian
  geometry \ Alexandrov space \ Spectral geometry  } 

\thanks{{\it 2000 Mathematics Subject Classification:}
Primary 53C23; Secondary 53C20, 58J53.}

\thanks{The author was supported in part by an NSF-AWM Mentoring Travel Grant.}

\date{March 1, 2011.}

\begin{abstract}
We show that any collection of $n$-dimensional orbifolds with sectional curvature and volume uniformly bounded below, diameter bounded above, and with only isolated singular points contains orbifolds of only finitely many orbifold homeomorphism types.  This is a generalization to the orbifold category of a similar result for manifolds proven by Grove, Petersen, and Wu.  It follows that any Laplace isospectral collection of orbifolds with sectional curvature uniformly bounded below and having only isolated singular points also contains only finitely many orbifold homeomorphism types.  The main steps of the argument are to show that any sequence from the collection has subsequence that converges to an orbifold, and then to show that the homeomorphism between the underlying spaces of the limit orbifold and an orbifold from the subsequence that is guaranteed by Perelman's stability theorem must preserve orbifold structure.
\end{abstract}

\maketitle

\vspace{5mm}
\section{Introduction}\label{introduction}

The interplay between the geometry and topology of manifolds has long been a theme of Riemannian geometry.  One type of question is that of \textit{finiteness}: by placing certain geometric bounds on a space, we would like to conclude that there are only a finite number of topological structures that the space can admit.

One of the earliest finiteness results was proven by Cheeger in 1970 \cite{Ch}.  He proved that for $n\geq 2$, the collection of closed $n$-dimensional manifolds with sectional curvature uniformly bounded above and below, diameter bounded above, and volume bounded below contains only finitely many diffeomorphism types.  In 1990, Grove, Petersen, and Wu generalized Cheeger's theorem to show that for $n\geq 5$, the upper bound on sectional curvature is unnecessary \cite{GPW}.  They showed that if $n\neq 3$, the class of closed $n$-dimensional manifolds with sectional curvature uniformly bounded below, diameter bounded above, and volume bounded below contains only finitely many homeomorphism types, and only finitely many diffeomorphism types if, in addition, $n\neq 4$. 

Perelman's powerful stability theorem, proven in 1991, states that every compact $n$-dimensional Alexandrov space with curvature bounded below by $\kappa$ has an associated value $\ve$ such that if $Y$ is another compact $n$-dimensional Alexandrov space with curvature bounded below by $\kappa$ that lies within Gromov-Hausdorff distance $\ve$ of $X$, then $X$ and $Y$ are homeomorphic.   (Perelman's original proof is given in \cite{P}.  His proof remained in preprint form and is not widely accessible, so Kapovitch gave a full account of the proof in \cite{K}.)  Since a manifold with sectional curvature bounded below is an example of an Alexandrov space, Perelman's theorem, together with Gromov's compactness theorem, imply Grove, Petersen, and Wu's homeomorphism finiteness result for all values of $n$. 

More recently, there has been developing interest in Riemannian orbifolds.  First introduced by Satake in 1956 under the name $V$-manifold \cite{S1}, an orbifold is a mildly singular generalization of a manifold that is locally modeled on $\mathbb{R}^n$ modulo the action of a finite group.  Since their introduction, orbifolds have been of increasing interest among geometers and topologists.  See, for example, \cite{Th}, \cite{Fu}, \cite{Bz}, \cite{F}, \cite{ALR}, \cite{DGGW}, \cite{KL}. In this paper we prove the following generalization of Grove, Petersen, and Wu's result to the orbifold category.  

\begin{theorem}
For any $\kappa\in \mathbb{R}$, $D>0$, $v>0$, and $n\in\mathbb{N}$, the collection of closed $n$-dimensional orbifolds having sectional curvature universally bounded below by $\kappa$, diameter bounded above by $D$, volume bounded below by $v$, and having only isolated singular points contains orbifolds of only finitely many orbifold category homeomorphism types.
\end{theorem}

In 1992, Brook, Perry, and Petersen extended Grove, Petersen, and Wu's result to collections of Laplace isospectral manifolds \cite{BPP}.  They showed that any isospectral collection of manifolds having a uniform lower bound on sectional curvature contains only finitely many homeomorphism types, and only finitely many diffeomorphism types if $n\neq 4$.   They did this by showing that the lower bound on sectional curvature combined with isospectrality implies an upper bound on diameter.   By Weyl's asymptotic formula, any isospectral collection of manifolds all have the same volume and dimension.  Thus the result follows directly.  

In \cite{St}, Stanhope proved that a lower bound on sectional curvature and isospectrality implies an upper bound on diameter for orbifolds as well.   As for manifolds, isospectral orbifolds have the same dimension and volume \cite{DGGW}.  Therefore we have the following.

\begin{corollary}
Any collection of Laplace isospectral orbifolds having a uniform lower bound on sectional curvature and only isolated singular points contains orbifolds of only finitely many orbifold category homeomorphism types.
\end{corollary}

The paper is organized as follows.  Section~\ref{background} outlines background information that is used in the paper.  In Section~\ref{noncriticalorbifoldcharts}, we prove the existence of a special type of orbifold chart above each of the singular points of the orbifolds in our collection.   The important feature of these charts is that each lies above a ball of a certain universally fixed radius about the singular point.   Section~\ref{finitepartition} explains how to partition the entire collection into a finite number of subcollections so that orbifolds in each subcollection have the same number and types of singular points.  Since there are only a finite number of subcollections, we may then restrict our arguments to a given subcollection.  In Section~\ref{limitisanorbifold}, we show that any sequence of orbifolds from a fixed subcollection must have a convergent subsequence, and that furthermore, the limit space carries an orbifold structure with the same number and types of singular points as orbifolds in the subsequence.  From Perelman's stability theorem, we know that there is a homeomorphism from the underlying space of the limit orbifold to the underlying space of any orbifold that is far enough along in the subsequence.  In Section~\ref{finiteorbifoldhomeotypes} we show that the topological homeomorphism guaranteed by the stability theorem is in fact an orbifold category homeomorphism.  We then conclude that each subcollection must have only had a finite number of orbifold homeomorphism types to begin with, thus proving our main result.

\subsection{Acknowledgments}  The author is very grateful to Elizabeth Stanhope for suggesting the problem and for helpful conversations.  The author also offers many thanks for Karsten Grove for his support on this project, and thanks to Vitali Kapovitch for clarifying conversation.

\vspace{5mm}
\section{Background}\label{background}

In this section we detail the definitions and notation that will be used in the following sections. 

\subsection{Riemannian orbifolds}

Roughly, an orbifold is a space that locally has the structure of $\mathbb{R}^n$ modulo the action of a finite group.  More concretely, we have the following.

\begin{definition}
Let $\xo$ be a second countable Hausdorff topological space.  Given an open set $U$ contained in $\xo$, an \textit{orbifold chart} over $U$ is a triple $(\wtu, \Gamma_U, \pi_U)$ such that:
\begin{enumerate}
\item $\wtu$ is a connected open subset of $\mathbb{R}^n$,
\item\label{groupaction} $\Gamma_U$ is a finite group that acts on $\wtu$ by homeomorphisms,
\item $\pi_U: \wtu \to U$ is a continuous map such that $\pi_U\circ\gamma = \pi_U$ for all $\gamma\in \Gamma_U$ and that induces a homeomorphism from $\wtu/\Gamma_U$ to $U$.
\end{enumerate} 
\end{definition}

A \textit{topological orbifold} is a topological space $\xo$ covered with a maximal atlas of orbifold charts subject to a compatibility condition (see page 2 in \cite{ALR}).  The space $\xo$ is called the \textit{underlying space} of the orbifold.  

An orbifold is \textit{smooth} if for each chart, $\Gamma_U$ acts by diffeomorphisms.  We define a Riemannian structure on a smooth orbifold by placing a $\Gamma_U$-invaviant Riemannian metric on the local cover $\wtu$ of each orbifold chart $\ccu$ and patching these local metrics together with a partition of unity.  A smooth orbifold with a Riemannian structure is called a \emph{Riemannian orbifold}.  
 
Let $p$ be a point in an orbifold $O$ and suppose that $\ccu$ an orbifold chart over a neighborhood $U$ of $p$.  If $\Gamma_U$ acts with nontrivial isotropy on a point $\tilde{p}$ in $\pi_U^{-1}(p)$, we say that $p$ is a \textit{singular point} of $O$.  A singular point $p$ is \textit{isolated} if there is a neighborhood about $p$ in $O$ containing no other singular points.  Since it can be shown that the isomorphism class of the isotropy group of a lift of $p$ is independent of both the choice of element in $\pi_U^{-1}(p)$ and the choice of orbifold chart about $p$, we unambiguously refer to the isomorphism class of the isotropy group as the \textit{isotropy type} of $p$.  

Suppose that $p$ is a singular point in a Riemannian orbifold $O$.  A \textit{distinguished chart of radius $\ve$} about $p$ is a chart $\ccu$ such that $\wtu$ is a convex geodesic ball of radius $\ve$ centered at $\tp$, $\pi(\tp)=p$, and the group $\Gamma_U$ fixes $\tp$.  Thus the isotropy group of $p$ is represented by $\Gamma_U$.  Notice that in this case, for every $\gamma\in\Gamma_U$, the differential $\gamma_*$ acts orthogonally on $T_{\tp}\wtu$.  We denote the group of differentials by $\Gamma_{U*}$.

The tangent bundle of an orbifold $O$ is denoted $TO$.   If $\ccu$ is a distinguished chart above a neighborhood $U$ of a point $p$ in $O$, the fiber of $TO$ above $p$, denoted $T_pO$, is given by $T_{\tilde p}\wtu/\Gamma_{U*}$.  The manifold exponential, $\exp_{\tilde p}$, maps a neighborhood $V$ of $0$ in  $T_{\tilde p}\wtu$ to $\wtu$.  Suppose that $v$ is an element of $T_pO$ with lift $\tilde v$ in $V$.   Define the orbifold exponential map by $\exp_p(v)=\pi_U(\exp_{\tilde p}(\tilde v))$.  Since $\Gamma_U$ acts by isometries, $\gamma\circ\exp_{\tilde p}=\exp_{\tilde p}\circ\gamma_*$ and thus the orbifold exponential is well defined.  For small $t$, $\exp_p(tv)$ is a length-minimizing geodesic emanating from $p$ in the direction of $v$.   As usual, the manifold exponential map restricts to a diffeomorphism between a neighborhood of $0$ in $T_{\tilde p}\wtu$ and a neighborhood of $\tilde p$ in $\wtu$.  Since $\gamma\circ\exp_{\tilde p}=\exp_{\tilde p}\circ\gamma_*$, the orbifold exponential map gives a homeomorphism between a neighborhood of $0$ in $T_pO$ and a neighborhood of $p$ in $U$.  

Various notions of measurement carry more or less directly from the manifold setting to the orbifold setting.  The diameter and volume of a Riemannian orbifold are defined as they are in the manifold setting.  We say that a Riemannian orbifold has sectional (resp. Ricci) curvature bounded below by $\kappa$ if each point in the orbifold can be locally covered by a manifold with sectional (resp. Ricci) curvature bounded below by $\kappa$.   Following the notation in \cite{GP}, we will denote the collection of all closed, compact $n$-dimensional orbifolds with sectional curvature universally bounded below by $\kappa$, diameter bounded above by $D$, and volume bounded below by $v$ with the symbol $\okdvn$.  We denote the subcollection of $\okdvn$ consisting of orbifolds having only isolated singular points by $\iokdvn$. 

Finally, we say that a function on an orbifold is smooth if at every point, it may be lifted to a smooth function on the local manifold cover above the point.  The Laplace operator acts on smooth functions on Riemannian orbifolds by acting on the local lifts of the function.  As for manifolds, the eigenvalue spectrum of the Laplace operator acting on smooth functions on a compact Riemannian orbifold is a discrete sequence of nonnegative numbers, tending to infinity (\cite{Chi}, \cite{DGGW}).  We say that two orbifolds are \textit{isospectral} if they have the same Laplace spectrum.

\subsection{Alexandrov spaces and distance functions}\label{alexandrovspaces}

Given a topological space, a length structure on the space is a rule that assigns lengths to admissible paths in the space.  For any space with a length structure, we define a metric (in the sense of a distance function) by saying that for points $p$ and $q$ in the space,
\[d(p,q) = \inf\{\text{Length}(c)  :  c \ \text{is an admissible path from} \ p \ \text{to} \ q\}.\]
A metric space whose metric arises from a length structure is called a length space.

In general, length spaces are much more complicated than manifolds or orbifolds, but by using the length structure, we can define the notion of a curvature bound on a length space.  There are multiple equivalent definitions of a bound on curvature, but all are based on making comparisons of distance functions arising from the length structure with distance functions on spaces of constant curvature.  Not all length spaces have a curvature bound.  Those that do, be it an upper or a lower bound, are called Alexandrov spaces.  There are well-defined notions of dimension, diameter, and volume for Alexandrov spaces.  We denote the collection of all $n$-dimensional Alexandrov spaces with curvature bounded below by $\kappa$, diameter bounded above by $D$, and volume bounded below by $v$ by $\akdvn$.  For a very nice exposition of length spaces and Alexandrov spaces, see \cite{BBI}.

In the orbifold setting, once we have a Riemannian metric in place, we may compute lengths of curves, and thus distances between points.  Thus any orbifold is a length space.  One can show that an orbifold with sectional curvature bounded below by $\kappa$ is in fact an Alexandrov space with a lower curvature bound $\kappa$ (see \cite{BBI}, Chapter 6 and \cite{Bz}, p.~42). Thus, we may place the central question of this paper in the setting of Alexandrov spaces of curvature bounded below, and use tools from that setting to analyze our problem.

If an orbifold $O$ is complete with respect to the metric arising from the length structure, any two points in $O$ can be joined by a curve that achieves the distance between them.  Such a curve, parametrized with respect to arclength, is called a \emph{segment} in $O$.   For more information on the length structure of an orbifold and behavior of segments, see \cite{Bz}. 

For fixed $p$ in a Riemannian orbifold, we shall consider the distance function from $p$, $d(p,\cdot)$.  In general, there may be more than one segment running from $p$ to $x$.  Let $G_p(x)$ denote the set of all unit vectors in $T_xO$ that are tangent to segments from $p$ to $x$.  Suppose that $\ccu$ is a distinguished orbifold chart above a neighborhood $U$ of $x$ with $\pi(\tilde x)=x$.  We say that $x$ is an \textit{$\alpha$-regular} point for $d(p,\cdot)$ if there is some vector $\tilde v\in T_{\tilde x}\wtu$ such that $\angle(\tilde v,\tilde w)< \alpha$ for all $\tilde w$ in the lift of  $G_p(x)$ in $T_{\tilde x}\wtu$.  We say that $x$ is \textit{$\alpha$-critical} for $d(p,\cdot)$ if no such vector exists.

\subsection{Gromov's compactness theorem}\label{compactnesstheorem}
We recall the definition of the Gromov-Hausdorff distance between two compact metric spaces.  

\begin{defn}
Let $A$ and $B$ be two closed subsets of a compact metric space $Z$.  The \textit{Hausdorff distance} between $A$ and $B$ is given by
\begin{equation*}
d_H^Z(A,B) = \inf\{\varepsilon\ \vert\ A\subset B(B,\varepsilon)\textrm{ and }B\subset B(A,\varepsilon)\},
\end{equation*}  
where $B(A,\varepsilon) = \{x\in Z\ \vert \ d(x, A) <\varepsilon\}$. 
\end{defn}

Gromov extended this to a definition of distance between to arbitrary compact metric spaces, which are not a priori necessarily subsets of a larger metric space.

\begin{defn} 
Suppose that $X$ and $Y$ are compact metric spaces.  The \textit{Gromov-Hausdorff} distance between $X$ and $Y$ is given by
\begin{equation*}
d_{GH} (X,Y)= \inf\{ d_H^Z (X,Y)\ \vert \ X \textrm{ and } Y \textrm{ are isometrically embedded in } Z\}
\end{equation*}
where the infimum is taken over all possible isometric embeddings of $X$ and $Y$ in $Z$ and $Z$ runs over all possible compact metric spaces.
\end{defn}

\begin{remark}\label{overallmetric}The Gromov-Hausdorff distance induces a metric on the collection of compact metric spaces, and therefore gives a notion of convergence for a sequence of compact metric spaces.  We have the following fact, which is useful when working concretely with convergent sequences (see, for example, \cite{Pe} p.278).   Suppose that $\{X_i\}$ is a sequence of compact metric spaces Gromov-Hausdorff converging to a compact metric space $X$.  Then there exists a metric on the disjoint union $X_i\sqcup_i X$ such that in this space the sequence $\{X_i\}$ converges to $X$ in the Hausdorff metric.
\end{remark}

Suppose that $X$ and $Y$ are two metric spaces.  We say that two maps $\phi,\psi:X\to Y$ are \textit{$\beta$-close} if $d_Y(\phi(x),\psi(x))<\beta$ for all $x\in X$.  A map $\theta: X\to Y$ is a $\beta$\textit{-Hausdorff approximation} if for all $y\in Y$ there exists $x\in X$ such that $d_Y(y,\theta(x))<\beta$ and for all $x_1, x_2\in X$, $\vert d_Y(\theta(x_1), \theta(x_2)) - d_X(x_1,x_2)\vert <\beta$.   Note that when two spaces are close to each other in the Gromov-Hausdorff metric, we can find a $\beta$-Hausdorff approximation from one to the other.

The following important theorem was originally proven in \cite{G}.  For the formulation below, see \cite{BBI}, Theorem 10.7.2.

\begin{theorem}[Compactness Theorem]\label{compactnesstheorem}
Let $\akdn$ denote the class of all Alexandrov spaces with dimension $\leq n$, curvature bounded below by $\kappa$, and diameter bounded above by $D$.  Relative to the Gromov-Hausdorff metric, $\akdn$ is compact.
\end{theorem}

One important corollary of the compactness theorem is that every sequence of spaces in $\akdn$ has a Gromov-Hausdorff convergent subsequence.

\subsection{Orbifold category homeomorphisms}\label{orbifoldhomeomorphism}

When we generalize the notion of a homeomorphism to the orbifold setting, our aim is to define the maps that preserve both the underlying topological structure of the orbifold as well as the singular data arising from the orbifold structure.  Throughout the paper, we will use the term \textit{homeomorphism} to refer to a map that preserves the underlying topological structure and the term \textit{orbifold homeomorphism} to refer to a map that preserves both the topological structure and the orbifold structure.

\begin{definition}\label{weakmap}
Let $O_1, O_2$ be orbifolds.  We say a map $f:O_1\to O_2$ is a \textit{continuous orbifold map} if $f$ is a continuous map from $X_{O_1}$ to $X_{O_2}$ and for any $x\in O_1$ there are orbifold charts $(\widetilde U, \Gamma_U, \pi_U)$ over a neighborhood $U$ of $x$ and $(\widetilde V, \Gamma_V, \pi_V)$ over a neighborhood $V$ of $f(x)$ such that:
\begin{enumerate}
\item $f(U) \subset V$,
\item there exists a continuous lift $\tilde f$ of $f$ carrying $\widetilde U$ to $\widetilde V$ for which $\pi_V \circ \tilde f = f \circ \pi_U$.
\end{enumerate}
\end{definition}

A continuous orbifold map $f:O_1\to O_2$ is a \textit{orbifold homeomorphism} if it is a homeomorphism of the underlying spaces $X_{O_1}$ and $X_{O_2}$ and the lift $\tilde f$ above each point is a homeomorphism from $\wtu$ to $\widetilde V$.  If $O_1$ and $O_2$ are smooth orbifolds and if for every $x\in O_1$, the lift $\tilde f$ is a smooth map, we say that $f$ is an \textit{orbifold diffeomorphism}.

\vspace{5mm}
\section{Noncritical orbifold charts}\label{noncriticalorbifoldcharts}

Suppose that $O\in\iokdvn$ and $p$ is a singular point $O$.  We construct a particular type of orbifold chart about $p$ that we refer to as a ``noncritical chart".  In Section~\ref{limitisanorbifold}, we will use these charts to prove that the Gromov-Hausdorff limit of a sequence of orbifolds in $\iokdvn$ is also an orbifold.

Before beginning the construction, we recall the following lemma (\cite{St} Lemma 8.2, \cite{PrSt} Lemma 3.1), slightly reformulated here.

\begin{lemma}\label{technicallemma}
Let $O\in\okdvn$.   There exist $r>0$ and $\alpha\in (0,\tfrac{\pi}{2})$ depending only on $\kappa, D, v$, and $n$ such that if $d(p,q)<r$, then $p$ is $\alpha$-regular for $d(q,\cdot)$ or $q$ is $\alpha$-regular for $d(p,\cdot)$.  
\end{lemma}

\begin{remark}\label{noncriticalball}  In the proof of Proposition 8.3 in \cite{St}, Stanhope showed that if $p$ is an isolated singular point, then for any point $q$ and any $\alpha\in(0,\tfrac{\pi}{2})$, $p$ is $\alpha$-critical for $d(q, \cdot)$.   Thus if $O$ is an orbifold in $\iokdvn$, the lemma guarantees the existence of values $r$ and $\alpha$ such that if $p$ is a singular point in $O$, then all points in $B(p, r)\backslash\{p\}$ are manifold points that are $\alpha$-regular for $d(p,\cdot)$.  The same values of $r$ and $\alpha$ work for all orbifolds in $\okdvn$.
\end{remark}

\begin{lemma}\label{vectorfield}  For a point $p$ in an orbifold $O\in\iokdvn$ and values $r>0$ and $\alpha\in (0,\tfrac{\pi}{2})$, suppose that all points in $B(p,r)\backslash\{p\}$ are $\alpha$-regular for the distance function $d(p,\cdot)$.  Then there is a smooth unit vector field $V$ on $B(p,r)\backslash\{p\}$ that acts as a generalized gradient vector field for $d(p,\cdot)$.  That is, for all $x\in B(p,r)\backslash\{p\}$ and all $w_x\in G_p(x)$, $\angle(V(x), w_x)<\alpha$.  If $\sigma$ is an integral curve for $V$, then for $s<t$, 
\begin{equation*}
d(p, \sigma(t))-d(p,\sigma(s))>\cos(\alpha)(t-s).
\end{equation*}
\end{lemma}
\medskip

\begin{proof}
By Remark~\ref{noncriticalball}, there are no other singular points in $B(p,r)\backslash\{p\}$.  Therefore, the proof of this lemma is identical to the proof of a similar theorem for manifolds found in \cite{Pe} (Proposition 11.1.2).  
\end{proof}

\begin{proposition}\label{noncriticalchartprop}
Let $p$ be a point in an orbifold $O\in\iokdvn$, $r>0$, and $\alpha\in(0,\tfrac{\pi}{2})$.  Suppose that all points in $B(p,r)\backslash\{p\}$ are $\alpha$-regular for $d(p,\cdot)$.  Let $\ve< r$.  Then there exists an orbifold chart above $B(p,\ve)$.
\end{proposition}

\begin{definition}
Charts constructed in the proof of Proposition~\ref{noncriticalchartprop} are called \textit{noncritical charts}.
\end{definition}

\begin{proof}[Proof of Proposition~\ref{noncriticalchartprop}]
Let $U = B(p,\ve)$ and let $\ccu$ denote the orbifold chart that we will construct.

By Lemma~\ref{vectorfield}, there exists a unit vector field $V$ on $B(p,r)\backslash\{p\}$ that acts as a generalized gradient vector field for $d(p,\cdot)$.   For $v\in S_pO$ and for small $t$, $\exp_p(tv)$ is a length-minimizing geodesic.  Thus, near $p$, integral curves of $V$ are given by $t\mapsto \exp_p(tv)$.  Furthermore, there is a value $\delta$ such that $\exp_p$ gives a homeomorphism from $B(O,\delta)\subset T_pO$ to $B(p,\delta)\subset O$.  Therefore, each $v\in S_pO$ gives rise to a unique integral curve $\sigma_v(t)$, equal to $\exp_p(tv)$ on $B(p,\delta)$.   For each such curve, $\lim_{t\to 0}\dot\sigma_v(t) = v$.

For $0<s<t$, Lemma~\ref{vectorfield} says that $d(p,\sigma_v(t))-d(p,\sigma_v(s))>\cos(\alpha)(t-s)$.  For $s$ close to $0$, $d(p,\sigma_v(s))=s$.  Therefore $d(p,\sigma_v(t)) - s(1-\cos(\alpha))>\cos(\alpha)t$.    Since $\alpha\in(0,\tfrac{\pi}{2})$, $1-\cos(\alpha)$ is positive.  This implies that $d(p,\sigma_v(t))>\cos(\alpha)t$.

By above, for each $v\in S_pO$ there is a value $t_v\in (\ve, (\cos(\alpha))^{-1}\ve)$ when $\sigma_v(t)$ first leaves $B(p,\ve)$.   This value varies continuously with $v$.

Let $\overline{U}\subset T_pO$ be the open neighborhood of $0$ defined by
\begin{equation*}
\overline{U} = \{tv\ \vert \ v\in S_pO\textrm{ and } 0\leq t< t_v\}.
\end{equation*}

Suppose that $\ccw$ is a distinguished orbifold chart above a neighborhood $W$ of $p$.  Since we can identify $T_pO$ with $T_{\tilde p}\wtw/\Gamma_{W*}$, consider $\overline{U}$ a subset of  $T_{\tilde p}\wtw/\Gamma_{W*}$ and lift to a set $\widetilde U\subset T_{\tilde p}\wtw$.  Because this is a lift, the action of $\Gamma_{W*}$ on $T_{\tilde p}\wtw$ restricts to an action on $\widetilde U$.   Let $\Gamma_U=\Gamma_{W*}$.  

Let $q: T_{\tilde p}\wtw\to T_pO$ denote the quotient map.  Define $\pi_U: \widetilde{U}\to B(p,\ve)$ by
\begin{align*}
\pi_U(0) &= p\\
\pi_U(t\tilde v) &= \sigma_{q(\tilde v)}(t).
\end{align*}
where $\tilde v$ is in $S_{\tilde p}\wtw$ and $0<t<t_{q(\tilde v)}$.

In order to conclude that $(\widetilde U, \Gamma_U, \pi_U)$ is an orbifold chart, we need to show that $\pi_U$ induces a homeomorphism $\bar{\pi}_U$ between $\overline{U} =q(\widetilde U)$ and $B(p,\ve)$.  Note that this will imply that $\pi_U$ is continuous since $\pi_U=\bar\pi_U\circ q$ and $q$ is continuous.

The map $\bar\pi_U: \overline{U}\to B(p,\ve)$ is defined by
\begin{align*}
\bar\pi_U(0) &= p\\
\bar\pi_U(tv) &=\sigma_v(t),
\end{align*}
where $v\in S_pO$ and $0<t<t_v$.  
On the neighborhood $B(p,\delta)$ of $p$, $\bar\pi_U=\exp_p$, so it is a local homeomorphism.   For $x\in B(p,\ve)\backslash B(p,\delta)$, suppose that $\sigma$ is the integral curve of $V$ through $x$.  Since distance along $\sigma$ increases proportional to a change in parameter, there is a value $s$ such that $B(p,\sigma(s))<\delta$.  By uniqueness of integral curves, $\sigma=\sigma_v$ for some $v\in S_pO$.  Therefore $\bar\pi_U$ is onto.  The map is one-to-one because integral curves do not intersect.    Finally, $\bar\pi_U$ is continuous on $\overline U\backslash B(0,\delta)$ as follows.  Since there are no singular points in $B(p,\ve)\backslash\{p\}$, $B(p,\ve)\backslash\{p\}$ is a manifold and $V$ is a smooth vector field on $B(p,\ve)\backslash\{p\}$.  The flow of a smooth vector field on a manifold acts by diffeomorphisms.  The map $\bar\pi_U$ can be decomposed into a radial flow towards $0$ in $\overline U$, followed by $\exp_p$, followed by a flow along $V$ away from $p$.   Provided the image of the radial flow in $\overline U$ is contained in $B(0,\delta)$, $\bar\pi_U$ is composed of continuous invertible maps, and therefore is continuous with continuous inverse itself.  
\end{proof}

\begin{corollary}\label{universalnoncriticalchart}
There is a universal value $\ve>0$ such that for any $O\in\iokdvn$ and singular point $p$ in $O$, there is an orbifold chart $\ccu$ above $B(p,\ve)$.  
\end{corollary}
\begin{proof}
This follows directly from Lemma~\ref{technicallemma}, Remark~\ref{noncriticalball}, and Proposition~\ref{noncriticalchartprop}.
\end{proof}

\vspace{5mm}
\section{A finite partition of $\iokdvn$.}\label{finitepartition}

We begin the process of showing that $\iokdvn$ contains orbifolds of only finitely many orbifold homeomorphism types by partitioning $\iokdvn$ into a finite number of subcollections based on the singular sets of the orbifolds.  Since there are only finitely many such subcollections, we will ultimately restrict our proof to an individual subcollection, which we will label $\mathcal{P}$.

From \cite{St}, we have the following two theorems.

\begin{theorem}\label{liz_singularpoints}
There is an upper bound, based only on $\kappa$, $D$, $v$, and $n$, on the number of singular points in any orbifold in the collection $\iokdvn$.
\end{theorem}

\begin{theorem}\label{liz_isotropy}
Suppose $\kappa\in\mathbb{R}$.  Let $\mathcal{S}$ be the collection of closed $n$-dimensional orbifolds with a lower bound $\kappa(n-1)$ on Ricci curvature, an upper bound $D>0$ on diameter, and a lower bound $v>0$ on volume.  There is an upper bound, based only on $\kappa$, $D$, $v$, and $n$, on the number of possible isotropy types, up to isomorphism, for points in an orbifold in $\mathcal{S}$.
\end{theorem}

We note that in \cite{St}, the second result is stated for an isospectral set of orbifolds having a uniform lower bound on Ricci curvature.  However, by \cite{DGGW}, all orbifolds in an isospectral set have the same dimension and volume.  In her paper, Stanhope proved that there is a upper bound on diameter for orbifolds in an isospectral set with a uniform lower bound on Ricci curvature.  The actual proof of the result in \cite{St} relies only on the bounds on Ricci curvature, diameter, volume, and dimension.  Thus, the proof of Theorem \ref{liz_isotropy} is the same as the proof of the corresponding result in \cite{St}.  

Since a lower bound on sectional curvature implies a lower bound on Ricci curvature, Theorem~\ref{liz_isotropy} implies that there are only finitely many possible isotropy types for singular points in orbifolds in $\okdvn$.

Using Theorems \ref{liz_singularpoints} and \ref{liz_isotropy}, we partition $\iokdvn$ into a finite number of subcollections so that all orbifolds in a given subcollection $\mathcal{P}$ have the same number and isotropy types of singular points.  For any orbifold in a fixed subcollection $\mathcal{P}$ and for any isotropy group $\Gamma$ associated with a point $p$ in the orbifold, there is an orthogonal group action by the corresponding group of differentials $\Gamma_*$ on the tangent space of the manifold cover $\wtu$ over $p$.   Since $\Gamma$ is finite, and since any finite group acting on $n$-dimensional space for fixed $n$ admits only finitely many representations, we may also assume that if $O_1$ and $O_2$ are two orbifolds in $\mathcal{P}$, with corresponding singular points $p_1$ and $p_2$, the actions of $\Gamma_*$ on the tangent spaces of $\wtu_1$ and $\wtu_2$ are isomorphic.

\vspace{5mm}
\section{An orbifold structure on a limit space}\label{limitisanorbifold}

Let $\mathcal{P}$ be a subcollection of $\iokdvn$ described in Section~\ref{finitepartition}.  Suppose that each orbifold in $\mathcal{P}$ has $m$ singular points, with associated isotropy groups $\Gamma^1,\Gamma^2,\dots,\Gamma^m$.  By Gromov's compactness theorem, the sequence $\{O_i\}$ has a convergent subsequence.  Pass to the subsequence and let $X$ denote the limit space.  Since there is a universal lower bound on the volumes of the $O_i$, there is no collapsing in the limit, so $\dim X=n$  (\cite{BBI}, Theorem 10.10.10).  Since each $O_i$ is compact, so is $X$ (\cite{BBI}, p.264).  We will prove that there is an orbifold atlas on $X$.  We begin with the following lemma.

\begin{lemma}\label{convergingpoints}
Let $X_i$ be a sequence of metric spaces converging in the Gromov-Hausdorff metric to a compact space $X$.  Suppose that for each $i$, $p_i$ is a point in $X_i$.  Then a subsequence of $\{p_i\}$ converges to a point in $X$.
\end{lemma}

\begin{proof}
If $X_i$ converges in the Gromov-Hausdorff metric to $X$, then there exists a metric on $X_i\sqcup_i X$ such that $X_i$ Hausdorff converges to $X$ in $X_i\sqcup_i X$ (see Remark~\ref{overallmetric}).  Thus for all $\varepsilon>0$, there exists $N$ such that $X\subset B(X_i,\varepsilon)$ and $X_i\subset B(X,\varepsilon)$ in $X_i\sqcup_i X$ for all $i\geq N$.  This implies that for large $i$, there exists $x_i\in X$ such that $d(p_i, x_i) <\varepsilon$.

Consider the sequence $\{x_i\}\subset X$.  $X$ is compact, so there is a convergent subsequence of $\{x_i\}$, converging to a point $x\in X$.  Relabel and call the subsequence $\{x_i\}$.   Now, by choosing $i$ large enough, both $d(p_i,x_i)<\ve$ and $d(x_i, x)<\ve$. Therefore,
\begin{equation*}
d(p_i, x)\leq d(p_i,x_i) + d(x_i, x) < \ve + \ve = 2\varepsilon.
\end{equation*}
Thus the sequence $\{p_i\}$ converges to $x$.  
\end{proof}

By applying Lemma~\ref{convergingpoints} $m$ times and passing repeatedly to subsequences, we may assume that each sequence of matched singular points $\{p_i\}$ in our sequence $\{O_i\}$ converges to a point $p$ in the limit space $X$.

Next, we record the following lemma (\cite{P2}, \cite{PP}, \cite{K2}, \cite{K}). 

\begin{lemma}\label{convexfunction}
Suppose $\{X_i\}\subset \akdvn$ is a sequence of $n$-dimensional Alexandrov spaces that Gromov-Hausdorff converges to an Alexandrov space $X$.  Let $p\in X$ and suppose that $\{p_i\in X_i\}$ converges to $p$.  Then there is a value $\delta> 0$, a 1-Lipschitz function $f:X\to\mathbb{R}$, and a sequence of 1-Lipschitz functions $f_i:X_i\to\mathbb{R}$ converging uniformly to $f$ such that $f$ is strictly convex in $B(p,\delta)$ and, for large $i$, $f_i$ is strictly convex on $B(p_i,\delta)$.  Furthermore, $f$ has a strict local minimum value of $0$ at $p$.  
\end{lemma}

The existence of the function $f$ was originally proven in \cite{P2} (Lemma 3.6).  Perelman constructed $f$ as a generalized average of distance functions and gave an argument that $f$ is 1-Lipschitz and is strictly concave in a neighborhood of radius $\delta$ about $p$.   (Note that $f$ is strictly concave if and only if $-f$ is strictly convex.  Thus, while the statements given \cite{P2}, \cite{PP}, and \cite{K2} concern strictly concave functions, the functions can be easily adjusted to give identical statements about strictly convex functions.)   In the proof of Lemma 4.3 in \cite{PP}, Perelman and Petrunin introduced the notion of lifting the function $f$ via Hausdorff approximations to functions $f_i$ on a sequence of Alexandrov spaces Gromov-Hausdorff converging to $X$.  They noted, by referring back to \cite{P2}, that for large $i$, the lifted functions are also 1-Lipschitz and strictly concave in a neighborhood of radius $\delta$ about $p_i$.   The original proof in \cite{P2} is not very detailed, but Kapovitch gave a complete proof of the strict concavity of both the function $f$ and the lifted functions $f_i$ in \cite{K2} (Lemma 4.2).  Although many of the results in \cite{K2} concern manifolds, the proof of Lemma 4.2 holds for Alexandrov spaces in general.  Kapovitch restated the result in \cite{K} (Lemma 3.1) in the form that most closely resembles the statement of Lemma~\ref{convexfunction} above.  The statement in \cite{K} is for manifolds.  After an application of Theorem 1 in \cite{GW} (see \cite{K2}, Theorem 2.2), the functions $f_i$ are assumed to be smooth.  While we cannot expect smoothness in the case of general Alexandrov spaces, for the reasons outlined here, the statement of Lemma~\ref{convexfunction} above is correct for Alexandrov spaces.

\begin{theorem}\label{limitorbifold}
For $n\geq 3$, let $\mathcal{P}$ be a subcollection of $\iokdvn$ described in Section~\ref{finitepartition}.  Suppose that $\{O_i\}$ is a sequence of orbifolds $\mathcal{P}$.  Then a subsequence of $\{O_i\}$ converges to a topological orbifold $X$.  
\end{theorem}

\begin{proof}
The proof of this result is based on a proof of a similar result for manifolds in \cite{K}.  We include the details here in order to explain what happens in the case of singular points.

By Gromov's compactness theorem and the uniform lower bound on volume, we know that a subsequence of $\{O_i\}$ converges to an $n$-dimensional Alexandrov space $X$.  By Lemma~\ref{convergingpoints}, we may assume that if $\{p_i\in O_i\}$ is a sequence of matched singular points, then it converges to a point $p$ in $X$.  Thus there are two types of points in $X$: those that are the limit of a sequence of matched singular points, and those that are not.  We construct orbifold charts for each type in turn. 

For a point $p\in O$, where $p=\lim_{i\to\infty}p_i$, let $\delta$, $f$, and $f_i$ be as in Lemma~\ref{convexfunction}.   Suppose first that $p$ is the limit of singular points. From Corollary~\ref{universalnoncriticalchart}, there is a value $\ve^{\prime}$ such that for each $i$, there is a noncritical orbifold chart $(\wtu_i, \Gamma, \pi_{U_i})$ above $B(p_i,\ve^{\prime})$.   We omit the subscript on $\Gamma$ since we are assuming that each $p_i$ has isomorphic isotropy.  Recall that $\ve^{\prime}$ is chosen so that $B(p_i,\ve^{\prime})\backslash\{p_i\}$ is a manifold.  Let $\dpp=\min\{\delta,\ve^{\prime}\}$.  By Theorem 1 in \cite{GW}, we may smooth $f_i$ on $B(p_i,\dpp)$, outside a neighborhood of $p_i$.  Call the new function $h_i$.   Note that since each $h_i$ is strictly convex, $h_i$ achieves a unique minimum value in $B(p_i,\dpp)$.

Choose $\ve$ small enough that $\{h_i\leq \ve\}\subset B(p_i,\dpp)$ for large $i$, $\{h_i\leq\ve\}$ is a convex, compact subset of $O_i$, and $\{f\leq\ve\}$ is a convex, compact subset of $X$.   Since $\{h_i\}$ converges uniformly to $f$, the sequence $\{h_i\leq\ve\}$ Gromov-Hausdorff converges to $\{f\leq \ve\}$.  Thus, by Theorem 1.2 in \cite{Pet}, the sequence of extremal sets $\{h_i=\ve\}$ Gromov-Hausdorff converges to $\{f=\ve\}$ with respect to their induced intrinsic metrics.  

Since $h_i$ is smooth away from $p_i$, $\{h_i=\ve\}$ is a submanifold of $B(p_i,\dpp)$.  As the boundary of a convex set, by the Gauss-Bonnet theorem, it has induced sectional curvature bounded below by $\kappa$.  Since $\{f=\ve\}$ is the limit of a sequence of Alexandrov spaces with curvature bounded below by $\kappa$, it is also an Alexandrov space with the same curvature bound.  Furthermore, since both $\{h_i=\ve\}$ and $\{f=\ve\}$ are boundaries of $n$-dimensional Alexandrov spaces, they have dimension $n-1$ (p.398, \cite{BBI}).  Therefore, by Perelman's stability theorem (\cite{P}, \cite{K}), we can choose a value of $i$ such that $\{h_i=\ve\}$ is homeomorphic to $\{f=\ve\}$.

Since $h_i$ and $f$ each have unique minimum in $B(p_i,\dpp)$ and $B(p,\dpp)$ respectively,  by Main Theorem 1.4 in \cite{P2}, $\{h_i<\ve\}$ is homeomorphic to an open cone over $\{h_i=\ve\}$ and $\{f<\ve\}$ is homeomorphic to an open cone over $\{f=\ve\}$.  Therefore $\{h_i<\ve\}$ is homeomorphic to $\{f<\ve\}$.   Notice that since $f$ has its unique minimum value at $p$, the point $p$ must be the cone point in $\{f<\ve\}$.  Furthermore, since $n\geq 3$ the point $p_i$ is the only point in $\{h_i<\ve\}$ with distinguished, non-manifold topology.  (Note that if $n=2$, the underlying topology of $p_i$ might still be that of a disk.)  Therefore, it must be the case that the homeomorphism from $\{h_i<\ve\}$ to $\{f<\ve\}$ maps $p_i$ to $p$.

By assumption, we have an orbifold chart $(\wtu_i,\Gamma,\pi_{U_i})$ above $B(p_i,\ve^{\prime})$ which we can restrict to an orbifold chart over $\{h_i<\ve\}$.  Composing this chart with the homeomorphism from $\{h_i<\ve\}$ to $\{f<\ve\}$ gives the desired orbifold chart above $\{f<\ve\}$.    

If $p$ is not the limit of singular points, let $2\beta$ be the distance between $p$ and the nearest point $q$ in $X$ that is the limit of a sequence of matched singular points $q_i$.  Then by applying the triangle inequality on $O_i\sqcup_i X$ under the metric described in Remark~\ref{overallmetric}, we see that $d(p_i,q_i)>\beta$ for large $i$.   Suppose again that $\delta$, $f$, and $f_i$ are as in Theorem~\ref{convexfunction}.   Suppose that $\beta^{\prime}<\min\{\delta,\beta\}$.   Then for large $i$ there are no singular points in $B(p_i,\beta^{\prime})$, so $B(p_i,\beta^{\prime})$ is a manifold.   Thus by Theorem 1 in \cite{GW}, we may approximate $f_i$ by a smooth, 1-Lipschitz, strictly convex function $h_i$ on $B(p_i,\beta^{\prime})$, where the sequence $\{h_i\}$ converges uniformly to $f$.  

From this point, the proof of the existence of a manifold chart about $p$ is exactly the same as the proof of Lemma 3.2 in \cite{K} and similar to the proof above of the existence of charts about the singular points.  The key point is to recognize that for small $\ve$, $\{h_i=\ve\}$ is in fact diffeomorphic to a sphere, and thus, by the stability theorem, $\{f=\ve\}$ is also homeomorphic to a sphere.  Thus, the set $\{f<\ve\}$ is an open cone over a sphere, and therefore is a manifold chart about $p$.  Note that in his proof, Kapovitch purposely did not use Perelman's stability theorem, so the proof can be simplified a bit, but the essence is the same.

Thus we have covered $X$ with an atlas of orbifold charts.  There are only two types of charts: those with an isolated singular point in the center, and those that are manifold charts.  The charts only overlap on manifold points, and therefore it is straightforward to check that the compatibility condition for orbifold charts is satisfied for this atlas.

\end{proof}

\vspace{5mm}
\section{Finitely many orbifold homeomorphism types}\label{finiteorbifoldhomeotypes}

We now prove the main results of the paper.

\begin{theorem}\label{maintheorem}
For any $\kappa\in \mathbb{R}$, $D>0$, $v>0$, and $n\in \mathbb{N}$, the collection $\iokdvn$ contains only finitely many orbifold homeomorphism types.  
\end{theorem}

\begin{remark}
We recall that in contrast with a homeomophism between the underlying space of two orbifolds, an orbifold homeomorphism takes into account both the topology of the underlying space and the orbifold structure on the space (cf. Section~\ref{orbifoldhomeomorphism}).
\end{remark}

\begin{proof}  

The case for $n=1$ follows from the fact that the only closed, compact, $1$-dimensional orbifolds are $S^1$ (with no singular points) and the interval $I$ where the endpoints are singular points with $\mathbb{Z}_2$ isotropy.  For $n=2$, the result was already proven in \cite{PrSt}.  From now on, we will assume that $n\geq 3$.

Following the argument in Section~\ref{finitepartition}, we partition $\iokdvn$ into a finite number of subcollections such that each subcollection contains orbifolds with identical singular data.  Suppose, by way of contradiction, that one of the subcollections, $\mathcal{P}$, contains orbifolds of infinitely many distinct orbifold homeomorphism types.  Let $\{O_i\}\subset\mathcal{P}$ be a sequence of orbifolds each having a distinct orbifold homeomorphism type.  By Theorem~\ref{limitorbifold}, after relabeling, a subsequence $\{O_i\}$ converges to an $n$-dimensional orbifold $O$ with curvature bounded below by $\kappa$, having the same singular data as the orbifolds in the sequence.  

Since an orbifold with a lower bound $\kappa$ on sectional curvature is an example of an Alexandrov space with lower bound $\kappa$ on curvature, Perelman's stability theorem implies that for large $i$, there is a homeomorphism $f_i$ from the underlying space of $O$ to the underlying space of $O_i$.  Suppose that $\{p_i\in O_i\}$ is a sequence of matched singular points with $p_i\to p$.  Recall that by construction of the orbifold atlas on $O$, $p\in O$ is a singular point.   Since $n\geq 3$, and since $p_i$ and $p$ are isolated nonmanifold points, they have distinct topology from all of their neighboring points.  Thus it must be the case that $f_i(p)=p_i$.   We will now show that $f_i$ is an orbifold homeomorphism, thereby contradicting the assumption that $\mathcal{P}$ contains infinitely many distinct orbifold homeomorphism types.

For a singular point $p\in O$ with $f_i(p) =p_i\in O_i$, let $(\wtu, \Gamma, \pi)$ be a chart above a neighborhood $U$ of $p$ and let $(\wtu_i,\Gamma,\pi_i)$ be a chart above the open neighborhood $U_i=f_i(U)$ about $p_i$.  Assume that $U_i$ contains no other singular points and $U$ contains no points that are a limit of a sequence of matched singular points.  We need a map $\tilde f_i:\wtu \to \wtu_i$ such that $\pi_i \circ \tilde f_i= f_i\circ\pi$.  

Consider $\unp$ and $\uinpi$.  Neither contains any singular points, so both are manifolds.  Since $f_i:U\to U_i$ and $f(p)=p_i$, the restriction $f_i:\unp\to\uinpi$ is a homeomorphism.  Define a map $\phi_i:\wtu\backslash\{\tp\}\to \uinpi$ by $\phi_i(\tilde y) = f_i(\pi(\tilde y))$.   Since $f_i$ is a homeomorphism and since $(\wtu\backslash\{\tp\},\pi)$ is a covering space of $\unp$, $(\wtu\backslash\{\tp\},\phi_i)$ is a covering space of $\uinpi$.  

Since $\wtu$ and $\wtu_i$ are both homeomorphic to $\mathbb{R}^n$ and $n\geq 3$, $\utnp$ and $\uitnpi$ are both simply connected.  Thus, $(\uitnpi,\pi_i)$ and $(\utnp,\phi_i)$ are both universal covering spaces of $\unp$.  This implies that there is a homeomorphism $\tilde \phi_i: \utnp\to\uitnpi$ such that $\pi_i\circ \tilde \phi_i= \phi_i$.  But this exactly says that $\pi_i\circ\tilde \phi_i= f_i\circ\pi$.

Define a map $\tilde f_i:\wtu\to\wtu_i$ by $\tilde f_i(\tilde y) = \tilde\phi_i(\tilde y)$ if $\tilde y\in\utnp$ and $\tilde f_i(\tilde p) = \tilde p_i$.  The map $\tilde f_i$ is continuous at $\tilde p$ because both $\tilde\phi_i$ and $f_i$ are homeomorphisms.  Thus we have a homeomorphism $\tilde f_i:\wtu\to \wtu_i$ such that $\pi_i\circ\tilde f_i=f_i\circ\pi$, so $f_i$ is in fact an orbifold homeomorphism as desired.  

\end{proof}

\begin{corollary}\label{spectralcorollary}
Let $\mathcal{IS}(\kappa)$ denote a collection of isospectral orbifolds with sectional curvature bounded below by $\kappa$ and having only isolated singular points.   Then $\mathcal{IS}(\kappa)$ contains orbifolds of only finitely many orbifold homeomorphism types.
\end{corollary}

\begin{proof}
Since the orbifolds are isospectral, they must all have the same volume and dimension \cite{DGGW}.  Moreover, since they share a uniform lower bound on sectional curvature, they also have a uniform lower bound on Ricci curvature.  Thus, by Proposition 7.4 in \cite{St}, the diameter of each orbifold in the collection is bounded above by some constant $D$.  The result then follows directly from Theorem~\ref{maintheorem}.  
\end{proof}

\vspace{10mm}

\end{document}